\documentclass[final,1p,times]{article}
\usepackage[utf8]{inputenc}

\usepackage{amsthm}
\usepackage{amsmath,amssymb,amsopn,amsfonts,mathrsfs,amsbsy,amscd}
\usepackage{varioref}
\usepackage{longtable}
\newcommand{\K}{{\cal{K}}}
\newcommand{\mk}{{\mathrm{k}}}
\newcommand{\md}{{\mathrm{d}}}

\newcommand{\J}{{\mathrm{J}}}
\newcommand{\lr}{\longrightarrow}

\newcommand{\om}{\omega}
\newcommand{\esp}{\quad\mbox{and}\quad}

\newcommand{\X}{{\cal X}}
\newcommand{\R}{\mathbb{R}}
\newcommand{\G}{{\mathfrak{g}}}

\newcommand{\ad}{{\mathrm{ad}}}
\newcommand{\tr}{{\mathrm{tr}}}

\newcommand{\B}{{\cal B}}

\newcommand{\na}{\nabla}

\newtheorem{theo}{Theorem}[section]
\newtheorem{pr}{Proposition}[section]
\newtheorem{Le}{Lemma}[section]
\newtheorem{co}{Corollary}[section]
\newtheorem{exem}{Example}
\newtheorem{remark}{Remark}
\newtheorem{conclu}{Conclusion}

\labelformat{theorem}{théorème~#1}
\labelformat{lemma}{lemme~#1}
\labelformat{proposition}{proposition~#1}

\title{Balanced Hermitian structures on  twisted cartesian products}
\author{M. W. Mansouri, A. Oufkou\\\\
mansouriwadia@gmail.com\\
 ahmed.oufkou@uit.ac.ma}

\begin{document}
	\maketitle
	\begin{abstract} We study Hermitian structures on the twisted cartesian product \newline  $(\G_{(\rho_1,\rho_2)},\J,\K)$  of two Hermitian Lie algebras according to two representations $\rho_1$ and $\rho_2$. We give the conditions on  $(\G_{(\rho_1,\rho_2)},\J,\K)$ to be balanced and locally conformally balanced. As an application,  we  classify  six-dimensional balanced Hermitian twisted cartesian products Lie algebras. 
	\end{abstract}
	
	\section{Introduction}
	The twisted product structure of Lie algebras defined by means of linear representations	is a well known construction which can be regarded as the generalization of the semidirect product of algebras, see for instance \cite{R}. By using this construction we can obtain examples of some special Hermitian metrics in Lie algebras.
	
	In literature, a Hermitian structure on a $2n$-dimensional smooth manifold $M$ is a pair $(\J,\K)$ where $\J$ is an integrable almost complex structure  which is compatible with a Riemannian metric $\K$ on $M$, namely $\K(\J .\,,\J .)=\K(.\,,.)$.
	The fundamental form is given by $\om(.\,,.)=\K(\J .\,,.)$
	and the Lee form is defined by $\theta=\J \md^*\om=-\md^*\om \circ \J$. A fundamental class of Hermitian metrics is provided by the K\"ahler metrics by means $\md\om=0$.
	\vfill
	\line(1,0){320}
	\newline
	2000 MSC. 22E25,\; 53C55,\; 53C15,\; 16W25.
	\newline
	\textit{key words}: Hermitian metrics,\, Balanced structures,\; Solvable Lie algebras,\; Twisted cartesian products.
	\newpage
	In literature, many generalizations of 	the K\"ahler condition have been introduced. Indeed, $(M,\J,\K)$ is called:
	\begin{itemize}
		\item Balanced  if $\theta=0$ .
		\item Locally conformally balanced (shortly, LCB) if $\md\theta=0$.
		\item Locally conformally K\"ahler (shortly, LCK) if $\md \om=\theta \wedge \om  $ where the Lee form $\theta$ is a closed 1-form.
	\end{itemize}	
	For general results about these generalized K\"ahler metrics, we refer the reader to \cite{AF}, \cite{B}, \cite{MG}.
	
	The aim of this paper is to study Hermitian structures on the twisted cartesian product  $(\G_{(\rho_1,\rho_2)},\J,\K)$. We give the conditions on  $(\G_{(\rho_1,\rho_2)},\J,\K)$ to be balanced, locally conformally balanced and K\"ahlerian.
	
	The article is organized as follows: In section $2$  we give the framework of the Hermitian Lie algebra  $(\G_{(\rho_1,\rho_2)},\J,\K)$. In particular, we provide a formula of the Lee form $\theta$ and we investigate it to know when the Hermitian twisted cartesian  product structure becomes balanced, LCB and K\"ahlerian. We also give  a method to construct balanced Hermitian Lie algebras from two Hermitian Lie algebras and apply it in section $3$.
	
	\textbf{Notation.}
	Using the Salamon notation, we write structure equations for Lie algebras:\newline e.g. $\mathfrak{rr}_{3,1}=(0, -12, -13, 0)$ fixing a coframe  $(e^1,e^2,e^3,e^4)$ for $\mathfrak{rr}_{3,1}^*$ that means $\md e^1=\md e^4=0$ and $\md e^2=-e^1\wedge e^2=-e^{12}$, $\md e^3=-e^1\wedge e^3=-e^{13}.$
	\section{Hermitian  twisted cartesian products}\label{section 2}
	Let $(\G_1,\J_1,\mk_1)$ and $(\G_2,\J_2,\mk_2)$ be two Hermitian Lie algebras for which there exist two linear Lie algebra representations
	\begin{equation*}
		\rho_1 : \G_1\lr Der(\G_2)\esp \rho_2 : \G_2\lr Der(\G_1).
	\end{equation*}
	We will say that $(\rho_1,\rho_2)$ is a representation compatible couple if  
	\begin{align}
		\rho_1(\rho_2(a)x)b&=\rho_1(\rho_2(b)x)a,\\
		\rho_2(\rho_1(x)a)y&=\rho_2(\rho_1(y)a)x,
	\end{align}	
	for any $x$, $y\in\G_1$ and $a$, $b\in\G_2$.
	
	The non-zero Lie brackets $[.,.]$ on $\G_{(\rho_1,\rho_2)}$ are defined on the vector space $\G_1 \oplus \G_2$ (direct sum of $\G_1$ and $\G_2$) by
	\begin{align*}
		[x,y]	&=[x,y]_1,\qquad\qquad\quad x,y\in\G_1\\
		[a,b]	&=[a,b]_2,\qquad\qquad\quad a,b\in\G_2\\
		[x,a]&=\rho_1(x)a-\rho_2(a)x,\quad x\in\G_1,\; a\in\G_2.
	\end{align*}
	
	A direct calculation shows that  $(\G_{(\rho_1,\rho_2)},[.\, ,\,.])$ is a Lie algebra if and only if $\rho_1$ and $\rho_2$ are compatible. The Lie algebra $(\G_{(\rho_1,\rho_2)},[.\, ,\,.])$  is called twisted cartesian
	product of $\G_1$ and $\G_2$ according to the representation compatible couple $(\rho_1,\rho_2)$ also noted $\G_1 \Join \G_2$ (see \cite{R} and \cite{V} for more details).
	\begin{remark} If $\rho_2=0$ then the twisted cartesian product becomes the semi-direct product of $\G_1$ and $\G_2$ by  $\rho_1$ and vice-versa. 
	\end{remark} 
	We let $\K$ be the scalar product on $\G_{(\rho_1,\rho_2)}$ which is defined by $\K(x+a,y+b)=\mk_1(x,y)+\mk_2(a,b)$, the  almost-complex structure $\J$ on $\G_{(\rho_1,\rho_2)}$ verify $\J(x+a)=\J_1(x)+\J_2(a)$	for $x,y\in\G_1$ and $a,b\in\G_2$.	Clearly, $\J$ and $\K$ are compatible i.e., $\K(\J .,\J .)=\K(.,.)$. The following proposition gives a necessary and sufficient condition for $\J$ to be integrable.
	\begin{pr}
		Let $(\G_1,\J_1,\mk_1)$ and $(\G_2,\J_2,\mk_2)$ be two Hermitian Lie algebras and let $(\rho_1,\rho_2)$  a representation compatible couple. Then, 	$(\G_{(\rho_1,\rho_2)},\om,\J)$ is a \textbf{Hermitian} Lie algebra if and only if
		\begin{align}
			[\rho_1(\J_1(x)),\J_2]&=\J_2\circ\rho_1(x)\circ\J_2+\rho_1(x),\qquad \forall x\in\G_1\label{com1}\\
			[\rho_2(\J_2(a)),\J_1]&=\J_1\circ\rho_2(a)\circ\J_1+\rho_2(a),\qquad \forall a\in\G_2. \label{com2}
		\end{align}
	\end{pr}\begin{proof}
		Using that $\J_1$ and $\J_2$ are integrable and the bilinearity of the Nijenhuis operator, the completeness of $\J$ reduces to $N_J(x,a)=0$ for all $x\in\G_1$ and $a\in\G_2$.
		\begin{align*}
			N_\J(x,a)&=[\J_1x,\J_2a]-[x,a]-\J([\J_1x,a]+[x,\J_2a])\\
			&=\rho_1(\J_1x)\J_2a-\rho_2(\J_2a)\J_1x-\rho_1(x)a+\rho_2(a)x\\&\quad-\J_2(\rho_1(\J_1x)a)+\J_1(\rho_2(a)\J_1x)-\J_2(\rho_1(x)\J_2a)+\J_1(\rho_2(\J_2a)x)\\
			&=\Big([\J_1,\rho_2(\J_2a)]x+\rho_2(a)x+\J_1(\rho_2(a)\J_1x)\Big)+\Big([\rho_1(\J_1x),\J_2]a\\ &\hspace*{3mm}-\rho_1(x)a-\J_2(\rho_1(x)\J_2a)\Big)
		\end{align*}
		and the result follows.		
	\end{proof}		
	\begin{remark} 	
		The conditions \eqref{com1} and \eqref{com2} are satisfied if $\rho_1(x)$ and $\J_2$  (resp; $\rho_2(a)$ and $\J_1$)  commute, for all $x\in\G_1$ (resp; $a\in\G_2$).
	\end{remark}	
	The Hermitian Lie algebra  $(\G_{(\rho_1,\rho_2)},\J,\K)$ is  called   Hermitian twisted cartesian product  of two Hermitian Lie algebras  according to two representations $\rho_1$ and $\rho_2$, or simply Hermitian twisted product if there is no confusion. Moreover, we consider in what follows  that $\rho_1(x)$ and $\J_2$  (resp; $\rho_2(a)$ and $\J_1$)  commute for all $x\in\G_1$ (resp; $a\in\G_2$).
	\begin{Le}\label{le cnx}  The Levi-Civita product associated to $(\G_{(\rho_1,\rho_2)},\K)$ is given by:	
		\begin{align*}
			\na_{x}a&=\frac{1}{2}((\rho_1(x)-\rho^*_1(x))a-(\rho_2(a)+\rho^*_2(a))x),\\
			\na_{a}x&=\frac{1}{2}((\rho_2(a)-\rho^*_2(a))x-(\rho_1(x)+\rho^*_1(x))a),\\
			\na_xy&=\overset{1}{\na}_xy, \quad
			\na_ab=\overset{2}{\na}_ab,
		\end{align*}
		for any $x,y\in \G_1$ and $a,b\in \G_2$, 	
	\end{Le}
	\begin{proof}
		The Koszul formula for the Levi-Civita connection in the invariant setting, is given by:
		\begin{align*}
			2\K(\na_{(x,a)}(y,b),(z,c))	&=\K([(x,a),(y,b)],(z,c))-\K([(y,b),(z,c)],(x,a))\\
			&\hspace*{3mm} -\K([(x,a),(z,c)],(y,b)).
		\end{align*}
		For $(x,a)=(x,0)$ and $(y,b)=(y,0)$ and $(z,c)=(z,0)$, we have
		\begin{align*}
			2\K(\na_{(x,0)}(y,0),(z,0))	&=\K([(x,0),(y,0)],(z,0))-\K([(y,0),(z,0)],(x,0))\\
			&\hspace*{3mm}-\K([(x,0),(z,0)],(y,0))\\
			&=\mk_1([x,y]_1,z)-\mk_1([y,z]_1,x)-\mk_1([x,z]_1,y)\\
			2\K(\na_{(x,0)}(y,0),(z,0))&=2\mk_1(\overset{1}{\na}_xy,z)
		\end{align*}
		and therefore
		\begin{equation*}
			\K(\na_xy,z)=\mk_1(\overset{1}{\na}_xy,z)=\K(\overset{1}{\na}_xy,z).
		\end{equation*}
		Hence $\na_xy=\overset{1}{\na}_xy$. A similar calculation show that
		\begin{equation*}
			2\K(\na_{(0,a)}(0,b),(0,c))=2\mk_2(\overset{2}{\na}_ab,c)
		\end{equation*}
		and therefore
		\begin{equation*}
			\K(\na_ab,c)=\mk_2(\overset{2}{\na}_ab,c)=\K(\overset{2}{\na}_ab,c).
		\end{equation*}
		Hence $\na_ab=\overset{2}{\na}_ab.$
		
		For $(x,a)=(x,0)$ and $(y,b)=(0,b)$, we have
		\begin{align*}
			2\K(\na_{(x,0)}(0,b),(z,c))	&=\K([(x,0),(0,b)],(z,c))-\K([(0,b),(z,c)],(x,0))\\
			&\hspace*{3mm}-\K([(x,0),(z,c)],(0,b))\\
			&=\K(\rho_1(x)b-\rho_2(b)x,(z,c))-\mk_1(\rho_2(b)z,x)\\
			&\hspace*{3mm}-\mk_2(\rho_1(x)c,b)\\
			&=\mk_2(\rho_1(x)b,c)-\mk_1(\rho_2(b)x,z)-\mk_1(\rho_2(b)z,x)\\
			&\hspace*{3mm}-\mk_2(\rho_1(x)c,b)\\
			&=\mk_2((\rho_1(x)-\rho^*_1(x))b,c)-\mk_1((\rho_2(b)+\rho^*_2(b))x,z).
		\end{align*}
		Replacing $(z,c)$ by $(z,0)$, then by $(0,c)$ we find
		\begin{equation*}
			2\na_xb=(\rho_1(x)-\rho^*_1(x))b-(\rho_2(b)+\rho^*_2(b))x.
		\end{equation*}
		A similar calculation show that
		\begin{equation*}
			2\na_ax=(\rho_2(a)-\rho^*_2(a))x-(\rho_1(y)+\rho^*_1(x))a
		\end{equation*}
		and we get	the desired result.
	\end{proof}	
	\begin{Le}\label{diff}
		The differential of the fundamental form $\omega$ associated to the twisted cartesian product  $(\G_{(\rho_1,\rho_2)},\J,\K)$ is given by:
		\begin{align*}
			\md\omega(x,y,c)&=-\om_1\big((\rho_2^*(c)+\rho_2(c))x,y \big),\\
			\md\omega(x,b,c)&=-\om_2\big((\rho_1^*(x)+\rho_1(x)\big) b,c),\\
			\md\omega(x,y,z)&=\md\omega_1(x,y,z) \esp \md\omega(a,b,c)=\md\omega_2(a,b,c),
		\end{align*}
		for any $x,y,z\in \G_1$ and $a,b,c\in \G_2$.
	\end{Le}
	\begin{proof}
		Using the exterior derivative of the 2-form $\omega$, we have
		\begin{align*}
			\md\omega(x,y,c)&=-\omega([x,y],c)+\omega([x,c],y)-\omega([y,c],x)\\
			&=-\omega([x,y]_1,c)+\omega(\rho_1(x)c-\rho_2(c)x,y)-\omega(\rho_1(y)c-\rho_2(c)y,x)\\
			&=\omega_1(\rho_2(c)y,x)-\omega_1(\rho_2(c)x,y)\\
			&=\mk_1\big((\J_1\circ\rho_2(c))y,x \big)-\mk_1\big((\J_1\circ\rho_2(c))x,y \big)\\
			&=\mk_1\big((\rho_2(c)\circ\J_1)y,x \big)-\mk_1\big((\J_1\circ\rho_2(c))x,y \big)\\
			&=\mk_1\big(J_1y,\rho_2^*(c)x \big)-\mk_1\big((\J_1\circ\rho_2(c))x,y \big)\\
			&=\mk_1\big(\rho_2^*(c)x,J_1y \big)-\mk_1\big((\J_1\circ\rho_2(c))x,y \big)\\
			&=-\mk_1\big((J_1\circ\rho_2^*(c))x,y \big)-\mk_1\big((\J_1\circ\rho_2(c))x,y \big)\\
			&=-\mk_1\Big(\J_1\circ\big(\rho_2^*(c)+\rho_2(c)\big) x,y \Big)\\
			\md\omega(x,y,c)&=-\om_1\big((\rho_2^*(c)+\rho_2(c))x,y \big).
		\end{align*}
		A similar calculation show that $\md\omega(x,b,c)=-\om_2\big((\rho_1^*(x)+\rho_1(x)\big) b,c).$		
	\end{proof}
	\begin{co}
		The Lie algebra  $(\G_{(\rho_1,\rho_2)},\J,\K)$ is 	\textbf{K\"ahlerian} if and only if  $(\G_i,\J_i,\mk_i)_{i=1.2}$  is  K\"ahlerian and  $\rho_i=-\rho_i^*$ for $i\in\{1,2\}$.
	\end{co}	
	\begin{proof}
		By Lemma \ref{diff} we have, $\md\omega(x,y,z)=\md\omega_1(x,y,z)$ and $\md\omega(a,b,c)=\md\omega_2(a,b,c)$ and since $\J_1$ commute with $\rho_2(c)$ we get :
		\begin{align*}
			\md\omega(x,y,c)&=-\omega_1(\rho_2(c)x,y)+\omega_1(\rho_2(c)y,x)\\
			&=\om_1(y,\rho_2(c)x)+\omega_1(\rho_2(c)y,x)\\
			&=\mk_1(\J_1y,\rho_2(c)x)+\mk_1(\J_1\circ \rho_2(c)y,x)\\
			&=\mk_1(\rho_2^*(c)\circ \J_1y,x)+\mk_1(\J_1\circ \rho_2(c)y,x)\\
			&=\mk_1 \Big(\big(\rho_2^*(c)\circ \J_1+\J_1\circ \rho_2(c)\big)y,x \Big)\\
			&=\mk_1 \Big(\big(\rho_2^*(c)\circ \J_1+ \rho_2(c)\circ\J_1\big)y,x \Big)\\
			\md\omega(x,y,c)&=\mk_1 \Big(\big((\rho_2^*(c)+ \rho_2(c))\circ\J_1\big)y,x \Big)
		\end{align*}
		and similarly since $\J_2$ commute with $\rho_1(x)$ we get
		\begin{align*}
			\md\omega(x,b,c)&=\mk_2 \Big(\big((\rho_1^*(x)+ \rho_1(x))\circ\J_2\big)c,b \Big).
		\end{align*}
		If we suppose that $(\G_{(\rho_1,\rho_2)},\J,\K)$ is K\"ahlerian then $\md\om_1=0=\md\om_2$, 	which means that $\om_1$ and $\om_2$ are both K\"ahlerian. In addition to that 
		\begin{align*}
			\md\omega(x,y,c)&=\mk_1 \Big(\big((\rho_2^*(c)+ \rho_2(c))\circ\J_1\big)y,x \Big)=0,\\
			\md\omega(x,b,c)&=\mk_2 \Big(\big((\rho_1^*(x)+ \rho_1(x))\circ\J_2\big)c,b \Big)=0,
		\end{align*}
		by  non-degeneracy of $\mk_1$ and $\mk_2$ and the fact that $\J_1^2=-Id_{\G_1}$ and $\J_2^2=-Id_{\G_2}$ thus	
		\[	\rho_1=-\rho_1^*,\esp
		\rho_2=-\rho_2^*.\]
		For the inverse it can be check easily.			
	\end{proof}
	
	We note that $\X_{\rho_1}$ (resp. $\X_{\rho_2}$) is the character of the representation  $\rho_1$ (resp. $\rho_2$) defined by $\X_{\rho_1}(x)=tr_2(\rho_1(x))$ and $\X_{\rho_2}(a)=tr_1(\rho_2(a))$. Our main result is the following theorem:
	\begin{theo}\label{theo1}
		Let $(\G_{(\rho_1,\rho_2)},\J,\K)$ be a Hermitian twisted product.
		Then, its associated Lee form $\theta$   is  given  by
		\begin{align}
			\theta(x)	&= \theta_1(x)-\X_{\rho_1}(x),\qquad x\in\G_1\\
			\theta(a)	&=\theta_2(a)-\X_{\rho_2}(a),\qquad a\in\G_2.
		\end{align}
		Moreover, the  Hermitian Lie algebra  $(\G_{(\rho_1,\rho_2)},\J,\K)$ is 
		\begin{itemize}
			\item[i.] \textbf{Balanced} if and only if   $\theta_1(x)=\X_{\rho_1}(x)$ and  $\theta_2(a)=\X_{\rho_2}(a)$ for all $x\in \G_1$ and $a\in \G_2$.
			\item [ii.] \textbf{LCB} if and only if $\G_1$ and $\G_2$ are both LCB and 
			\begin{align*}
				\theta_1\big(\rho_2(a)x\big)-\theta_2\big(\rho_1(x)a\big)=\X_{\rho_1}(\rho_2(a)x\big)-\X_{\rho_2}(\rho_1(x)a\big)
			\end{align*}
			for all $x\in \G_1$ and $a\in \G_2$.\end{itemize}\end{theo}
	\begin{proof}
		\begin{itemize}
			\item[i.] Let $\B_1=\{e_1,\dots,e_{2n_1}\}$ and $\B_2=\{f_1,\dots,f_{2n_2}\}$ an orthonormal basis of $\G_1$ and $\G_2$, respectively. 
			Recall the definition of $\md^*\om$ :
			\begin{align*}
				\md^*\om(X)=-\sum_{i=1}^{2n_1}(\nabla_{e_i}\om)(e_i,X)-\sum_{j=1}^{2n_2}(\nabla_{f_j}\om)(f_j,X)
			\end{align*}
			where $X\in \G_{(\rho_1,\rho_2)}$ and $\nabla$ is the Levi-Civita connection associated to $(\G_{(\rho_1,\rho_2)},\K)$. Taking into consideration the Koszul formula for the Levi-Civita connection, we get
			\begin{align*}
				(\nabla_{e_i}\om)(e_i,X)
				&=-\om(\na_{e_i}e_i,X)-\om(e_i,\na_{e_i}X)\\
				&=\K(\na_{e_i}e_i,\J X)-\K(\J e_i,\na_{e_i}X)\\
				&=\frac{1}{2} \Big(2\K([\J X,e_i],e_i)+\K([X,e_i],\J e_i)+\K([X,\J e_i],e_i)\\
				&\hspace*{3mm}-\K([\J e_i,e_i],X) \Big)
			\end{align*}
			Therefore
			\begin{align*} -\sum_{i=1}^{2n_1}(\nabla_{e_i}\om)(e_i,X)&=-\tr_1(\ad_{\J X})-\frac12 \Big(-\tr_1(\J \ad_X)+\tr_1(\ad_X J)\\
				&\hspace*{3mm}- \sum_{i=1}^{2n_1}\K([\J_1e_i,e_i]_1,X) \Big)\\
				&=-\tr_1(\ad_{\J X})+\frac12 \sum_{i=1}^{2n_1}\K([\J_1e_i,e_i]_1,X).
			\end{align*}
			Similarly, we get
			\begin{align*}
				-\sum_{j=1}^{2n_2}(\nabla_{f_j}\om)(f_j,X)=-\tr_2(\ad_{\J X})+\frac12 \sum_{j=1}^{2n_2}\K([\J_2f_j,f_j]_2,X).
			\end{align*}
			So
			\begin{align*}
				\md^*\om(X)=-\tr(\ad_{\J X})+\frac12 \Big(\sum_{i=1}^{2n_1}\K([\J_1e_i,e_i]_1,X)+\sum_{j=1}^{2n_2}\K([\J_2f_j,f_j]_2,X)\Big).
			\end{align*}
			We have for $X=e_k,\, k=1 \ldots 2n_1$
			\begin{align*}
				\md^*\om(e_k)=&-\tr\Big(\ad_{\J_1 e_k}\Big)+\frac12\sum_{i=1}^{2n_1}\mk_1\Big([\J_1e_i,e_i]_1,e_k\Big)\\
				&=-\sum_{i=1}^{2n_1}e^i\Big([\J_1e_k,e_i]_1\Big)-\sum_{i=1}^{2n_2}f^i\Big([\J_1e_k,f_i]\Big)\\
				&\hspace*{3mm}+\frac12\sum_{i=1}^{2n_1}\mk_1\Big([\J_1e_i,e_i]_1,e_k\Big)\\
				&= \md^*\om_1(e_k)-\sum_{i=1}^{2n_2}f^i\Big(\rho_1(\J_1e_k)f_i\Big)+\sum_{i=1}^{2n_2}f^i\Big(\rho_2(f_i)\J_1e_k\Big)\\
				&= -\theta_1\circ\J_1^{-1}(e_k)-\tr_2\Big(\rho_1(\J_1e_k)\Big).
			\end{align*}
			So
			\begin{align*}
				\md^*\om(\J_1 e_k)&=-\theta_1(e_k)-\tr_2\big(\rho_1(-e_k)\big)\\
				&=-\theta_1(e_k)+\tr_2\Big(\rho_1(e_k)\Big)\\
				&=-\theta_1(e_k)+\X_{\rho_1}(e_k).
			\end{align*} 
			Since $\theta=-\md^*\om \circ \J$, we get
			\begin{align*}
				\theta(e_k)&=-\md^*\om \circ \J(e_k)\\
				&=-\md^*\om(\J_1 e_k)\\
				\theta(e_k)&=\theta_1(e_k)-\X_{\rho_1}(e_k).
			\end{align*}
			The same is true for the second identity $\theta(f_r)=\theta_2(f_r)-\X_{\rho_2}(f_r)$ that we get  for $X=f_r,$ 
			
			$ r=1,\ldots 2n_2$.	Taking into consideration the balanced condition, we get the result.
			\item[ii.] By the assertion $(i)$, we know that 
			\begin{align*}
				\theta(x)	&= \theta_1(x)-\X_{\rho_1}(x),\\
				\theta(a)	&=\theta_2(a)-\X_{\rho_2}(a).
			\end{align*} 
			for all $x\in \G_1$ and $a\in\G_2$.	So the differential of $\theta$ is defined by
			\begin{align*}
				\md\theta(x,y)&=\md\theta_1(x,y),\quad \forall x,y \in \G_1\\
				\md\theta(a,b)&=\md\theta_2(a,b),\quad \forall a,b \in \G_2\\
				\md\theta(x,a)&=\theta_1\big(\rho_2(a)x\big)-\theta_2\big(\rho_1(x)a\big)-\Big(\X_{\rho_1}(\rho_2(a)x\big)-\X_{\rho_2}(\rho_1(x)a\big)\Big),
			\end{align*}
			where  $( x,a)\in  \G_1\times \G_2$.\newline
			Then  $\md\theta=0$ if and only if $\md\theta_1=\md\theta_2=0$ and 
			\[\theta_1\big(\rho_2(a)x\big)-\theta_2\big(\rho_1(x)a\big)=\X_{\rho_1}(\rho_2(a)x\big)-\X_{\rho_2}(\rho_1(x)a\big)\]	and the result follows.
	\end{itemize}\end{proof} 	
	\begin{remark}			
		If $(\G_1,\om_1,\J_1)$ and $(\G_2,\om_2,\J_2)$ are balanced then the Hermitian  twisted cartesian product is balanced  if and only if $
		\tr_2\big(\rho_1(x))=0$ and $\tr_1(\rho_2(a))=0$, for all $x\in\G_1$ and $a\in\G_2$.	
	\end{remark}
	A simple case that generates several examples is when $\G_1$ and $\G_2$ are abelians.
	\begin{co}\label{co1}
		Let $(\G_1,\J_1,\mk_1)$ and $(\G_2,\J_2,\mk_2)$ be two abelian Hermitian Lie algebras and $(\G_{(\rho_1,\rho_2)},\J,\K)$  their Hermitian twisted cartesian product.	The associated Lee form $\theta$   is  given  by
		\[\theta(x)=-\X_{\rho_1}(x) \esp \theta(a)=-\X_{\rho_2}(a).\]	
		Moreover, the Lie algebra  $(\G_{(\rho_1,\rho_2)},\om,\J)$ is \textbf{balanced} if and only if $\X_{\rho_1}(x)=\X_{\rho_2}(a)=0$	and {LCB} if and only if $tr_1\big(\rho_2(\rho_1(x)a)\big)=tr_2\big(\rho_1(\rho_2(a)x)\big).$
	\end{co}
	\begin{exem} 
		A general result of the corollary above is the  balanced Hermitian twisted cartesian product  $\left(\R^{2p}\Join \R^{2q},\J,\K\right)$ 
		such that \newline $\left(\R^{2p}=span\{e_1,\dots,e_{2p}\},\J_1,\mk_1\right)$ and $\left(\R^{2q}=span\{f_1,\dots,f_{2q}\},\J_2,\mk_2\right)$ are the two abelian Hermitian Lie algebras  associated respectively to the representations \newline  $\rho_1:\R^{2p}\rightarrow \text{End}(\R^{2q})$  and  $\rho_2:\R^{2q}\rightarrow \text{End}(\R^{2p})$ defined by:
		\newline $\rho_1(e_i)=\left (
		\begin{array}{c|c}
			A_i & B_i\\
			\hline
			-B_i & A_i\\ 
		\end{array}
		\right)$ and  $\rho_2(f_j)=\left (
		\begin{array}{c|c}
			C_j & D_j\\
			\hline
			-D_j & C_j\\ 
		\end{array}
		\right)$\newline  where $(A_i,B_i) \in \text{Diag}(a_1,\dots,a_q)\times \text{Diag}(b_1,\dots,b_q)$ and\newline  $(C_j,D_j) \in \text{Diag}(c_1,\dots,c_p)\times \text{Diag}(d_1,\dots,d_p)$ and $\tr\, A_i=0=\tr\, C_j $ for  the standard complex structures  $\J_1=\begin{pmatrix}
			0&-I_p\\
			I_p&0\\
		\end{pmatrix}$  and $\J_2=\begin{pmatrix}
			0&-I_q\\
			I_q&0\\
		\end{pmatrix}$.
	\end{exem}
	\section{Applications in dimension six}
	In this section we look for six-dimensional balanced  Hermitian twisted cartesian products. In order to do that only two cases are presented, the first one is $\R^2 \Join \G$ and the second one is $\mathrm{aff(\R)}\Join\G$, where $\dim(\G)=4$.

	\subsection{$\mathbf{\mathbb{R}^2\Join\G}$  with  $\dim(\G)=4$}
	Let  $\R^2=span \{e_1,e_2\}$ with $[e_1,e_2]=0$, $\om_1=e^{12}$ and  $\J_1(e_1)=e_2$ and $(\mathfrak{g},\om_2,\J_2)$ be a four-dimensional Hermitian Lie algebra and its associated Lee form $\theta_2$, 
	in order to applicate our Theorem \ref{theo1} $\mathfrak{g}$ is necessary Lck (because $\md\theta_2=\md \,\tr(\rho_2)=0$ and $\md \om_2=\theta_2 \wedge \om_2$). In the other hand, the article \cite{DA} gives a classification of Lck structures on four dimensional Lie algebras up to linear equivalence. 
	In fact basing on results of Table 2 see \cite{DA}, we twist $\mathbb{R}^2$ with each one of the four-dilensional  Lie algebras listed in this table. We obtain the following theorem. 
	\begin{theo}\label{Th3.1}
		Balanced Hermitian twisted cartesian products Lie algebras of type $(\R^2\Join\G,\omega,\J)$ where $\dim(\mathfrak{g)}=4$ are described as follows:		
		\begin{itemize}
			\item $\R^2 \rtimes \mathfrak{rr}_{3,1}:$
			\begin{itemize}\item []
			$(-13-x23-t26,x13+t16-23,0,-34,-35,0)$
				\item[]$\J(e_1)=e_2$,  $\J(e_3)=e_6$, $\J(e_4)=e_5$.
				\item[]$\om=e^{12}+\sigma e^{36}+e^{45}$ with $\sigma>0$.
			\end{itemize}
			\item $\R^2 \rtimes \mathfrak{rr}_{3,0}:$
			\begin{itemize} 
				\item[] {\small $\left(\frac{\delta}{2}13+\frac{\sigma}{2\delta}15-x23-z25-t26,x13+z15+t16+\frac{\delta}{2}23+\frac{\sigma}{2\delta}25,0,-34,0,0\right)$}
			
				\item[]$\J(e_1)=e_2$,  $\J(e_3)=e_4$, $\J(e_5)=e_6$.
				\item[]$\om=e^{12}+\frac{\delta(\delta+1)}{\sigma} e^{34}+e^{36}-e^{45}+\frac{\sigma}{\delta^2}e^{56}$ with  $\delta>0$, $\sigma>0$.
			\end{itemize}
			\item $\R^2 \rtimes \mathfrak{rh}_{3}:$ 
				\begin{itemize}\item []
			$(-\frac12 16-x23-y24-t26,x13+y14+t16-\frac12 26,0,0,-34,0)$
				\item [] $\J(e_1)=e_2,\, \J(e_3)=e_4,\, \J(e_5)=e_6$.
				\item [] $\om=e^{12}+\sigma \left(e^{34}+e^{56}\right) $ with $\sigma>0.$
			\end{itemize}
			\item $\R^2 \rtimes \mathfrak{rr}'_{3,\gamma}$\quad $\gamma>0:$
			\begin{itemize}
			\item[]$(-\gamma13-x23-t26,x13+t16-\gamma23,0,-\gamma34-35,34-\gamma35,0)$
				\item [] $\J(e_1)=e_2,\, \J(e_3)=e_6,\, \J(e_4)=\pm e_5$.
				\item [] $\om=e^{12}+\sigma e^{36}+e^{45} $ with $\sigma>0.$
			\end{itemize}
			\item $\R^2 \rtimes_1 \mathfrak{r}_2\mathfrak{r}_2:$
						\begin{itemize}\item[]
			 $\left(-\frac12 13-x23-z25,x13-\frac12 23+z15,0,-34,0,-56 \right)$
				\item [] $\J(e_1)=e_2,\, \J(e_3)=e_4,\, \J(e_5)=e_6$				
				\item [] $\om=e^{12}+\omega_{34} e^{34}+\omega_{56} (-e^{36}+e^{45}+e^{56})$ with $\om_{34}>\om_{56}>0$.
			\end{itemize}			
			\item $\R^2 \rtimes_2 \mathfrak{r}_2\mathfrak{r}_2:$
			\begin{itemize}\item []
			 $\left(-\frac12 15-x23-z25,x13+z15-\frac12 25,0,-34,0,-56 \right)$
		
				\item [] $\J(e_1)=e_2,\, \J(e_3)=e_4,\, \J(e_5)=e_6$.
				\item []$\om=e^{12}+\om_{34}(e^{34}-e^{36}+e^{45})+\om_{56}e^{56}$ with $\om_{56}>\om_{34}>0.$
			\end{itemize}
			\item $\R^2 \rtimes_3 \mathfrak{r}_2\mathfrak{r}_2:$
				\begin{itemize}\item []
			 $(\frac{\sigma}{2}13+\frac{\tau}{2}15-x23-z25,x13+z15+\frac{\sigma}{2}23+\frac{\tau}{2}25,0,-34,0,-56)$
		
				\item [] $\J(e_1)=e_2,\, \J(e_3)=e_4,\, \J(e_5)=e_6$.
				\item [] $\om=e^{12}+\mu \left(\frac{1+\sigma}{\tau}e^{34}+e^{36}-e^{45}+\frac{\tau+1}{\sigma}e^{56}\right)$ with $\sigma\tau \neq0,\, \sigma+\tau\neq -1,\, \mu \neq 0,\, \frac{\mu(1+\sigma)}{\tau}>0,\, \frac{\mu(1 + \tau)}{\sigma}>0,\,  \frac{\sigma+ \tau+1}{\sigma\tau}>0.$
			\end{itemize}
			\item $\R^2 \rtimes_1 \mathfrak{r}'_2:$
			\begin{itemize}\item []
			 $(-13-x23-y24,x13+y14-23,0,0,-35+46,-36-45)$
			
				\item [] $\J(e_1)=e_2,\, \J(e_3)=e_5$, $\J(e_4)=e_6$.
				\item [] $\om=e^{12}+\om_{34} (e^{34}+e^{56})+\om_{35}(e^{35}+e^{46})$ with $\om_{35}>0,\, \om_{35}^2-\om_{34}^2>0,\, \om_{34}>0$.
			\end{itemize}
			\item $\R^2 \rtimes_2 \mathfrak{r}'_2:$
			\begin{itemize}\item []
			 $\left(\frac{\alpha}{2}13-x23-y24,x13+y14+\frac{\alpha}{2}23,0,0,-35+46,-36-45 \right)$
			
				\item [] $\J(e_1)=e_2,\, \J(e_3)=e_5$, $\J(e_4)=e_6.$
				\item [] $\om=\om_{46}(-(\alpha+1)e^{35}+e^{46})$ with $\om_{46}>0,\, \alpha+1<0.$ 
			\end{itemize}
			\item $\R^2 \rtimes_3 \mathfrak{r}'_2:$
			\begin{itemize}\item []
			 $(-13-x23-y24,x13+y14-23,0,0,-35+46,-36-45)$
			
				\item []$\J(e_1)=e_2,\,\J(e_3)=-ae_3+\frac{a^2 +1}{b}e_4$, $\J(e_5)= e_6$ with $(a,b)\neq (0,1)$.
				\item[] $\om=e^{12}+\om_{34}e^{34}+e^{56}$ with $b\,\om_{34}>0.$
			\end{itemize}
			\item $\R^2 \ltimes_4 \mathfrak{r}'_2:$
			\begin{itemize}\item []
			 $(-13-24,14-23,0,0,-a(13+24)+b(23-14)-35+46,a(23-14)+b(13+24)-36-45)$
			
				\item[] $\J(e_1)=e_2,\,\J(e_3)=e_4$, $\J(e_5)= e_6$  .
				\item[] $\om=e^{12}+\sigma e^{34}+\om_{56}e^{56}$ with $\sigma>0.$
			\end{itemize}
			\item $\R^2 \rtimes_5 \mathfrak{r}'_2:$
			\begin{itemize}\item []
			 $(-13-x23-y24,x13+y14-23,0,0,-35+46,-36-45)$
			
				\item[] $\J(e_1)=e_2,\,\J(e_3)=e_4$, $\J(e_5)= e_6$.
				\item[] $\om=e^{12}+\sigma e^{34}+\om_{56}e^{56}$ with $\sigma>0.$
			\end{itemize}
			\item $\R^2 \ltimes_6 \mathfrak{r}'_2:$
			\begin{itemize}\item []
			 $(-13+24,-14-23,0,0,a(24-13)-b(14+23)-35+46,-a(14+23)+b(13-24)-36-45)$
			
				\item[] $\J(e_1)=e_2,\,\J(e_3)=e_4$, $\J(e_5)= e_6.$
				\item[] $\om=e^{12}+\sigma e^{34}+\om_{56}e^{56}$ with $\sigma>0.$
			\end{itemize}
			\item $\R^2 \rtimes_7 \mathfrak{r}'_2:$
			\begin{itemize}\item []
			 $(\frac{\alpha}{2}13+\frac{\beta}{2}14-x23-y24,x13+y14+\frac{\alpha}{2}23+\frac{\beta}{2}24,0,0,-35+46,-36-45)$
			
				\item [] $\J(e_1)=e_2,\, \J(e_3)=e_5,\, \J(e_4)=e_6.$
				\item[] $\omega=e^{12}+\om_{35}\left(e^{35}+\frac{\beta}{\alpha}(e^{36}+e^{45})+\frac{\beta^2-\alpha}{\alpha(\alpha+1)}e^{46}\right)$ with $\om_{35}>0,\newline \alpha\neq -1,\, \alpha\neq0, \beta>0,\, \frac{\beta^2-\alpha}{\alpha(\alpha+1)}>\frac{\beta^2}{\alpha^2}.$
			\end{itemize}
			\item $\mathbb{R}^2 \rtimes \mathfrak{r}_{4,\alpha,1}$ ,\; $ \alpha\notin \left\{0;1\right\}:$
			\begin{itemize} \item []
			 $(-16-t25,t16-26,36,\alpha46,56,0)$
			
				\item[]  $\J(e_1)=e_2,\, \J(e_3)=e_5$, $\J(e_4)=-e_6.$
				\item[] $\om=e^{12}+e^{35}+\sigma e^{46}$ with $\sigma<0.$ 
			\end{itemize}
			\item $\mathbb{R}^2 \rtimes \mathfrak{r}_{4,\alpha,\alpha}$ ,\; $ \alpha\notin \left\{0;1\right\}:$
			\begin{itemize} \item []
			 $(-\alpha16-t26,t16-\alpha26,36,\alpha46,\alpha56,0)$
			
				\item[] $\J(e_1)=e_2,\, \J(e_3)=-e_6$, $\J(e_5)=-e_4.$
				\item[] $\om=e^{12}+\sigma e^{36} + e^{45}$ with $\sigma<0.$
			\end{itemize}
			\item $\mathbb{R}^2 \rtimes \mathfrak{r}'_{4,\gamma,\delta}$ ,\; $\delta>0,\, \gamma \neq0:$ 
			\begin{itemize} \item []
			 $(-\gamma16-t26,t16-\gamma26,36,\gamma46+\delta56,-\delta46+\gamma56,0)$
			
				\item[] $\J(e_1)=e_2,\, \J(e_3)=-e_6,\, \J(e_5)=\pm e_4$
				\item[] $\om=e^{12}+\sigma e^{36}\pm e^{45}$ with $\sigma<0.$
			\end{itemize}
			\item $\R^2 \rtimes \mathfrak{d}_4:$ 
			\begin{itemize} \item []
			$\left(-\frac12 16-t26,t16-\frac12 26,36,-46,-34,0\right)$
			
				\item[] $\J(e_1)=e_2,\, \J(e_3)=-e_5,\, \J(e_4)=-e_6.$
				\item[] $\om=e^{12}- e^{35}+ \sigma e^{46}$ with $\sigma<0.$
			\end{itemize}
			\item $\R^2 \rtimes_1 \mathfrak{d}_{4,1}:$
			\begin{itemize} \item []
			 $\left(-\frac12 16-y24-t26,y14+t16-\frac12 26,36,0,-34+56,0\right)$
			
				\item[] $\J(e_1)=e_2,\, \J(e_3)=e_6,\, \J(e_4)=e_5.$
				\item[] $\om=e^{12}+\sigma e^{36}+  e^{45}$ with $\sigma>0.$
			\end{itemize}
			\newpage
			\item $\R^2 \rtimes_2 \mathfrak{d}_{4,1}:$ 
			\begin{itemize}\item []
			$\left(\frac12 14+\frac{\alpha}{2}16-y24-t26,y14+t16+\frac12 24+\frac{\alpha}{2}26,36,0,-34+56,0\right)$
			
				\item[] $\J(e_1)=e_2,\, \J(e_3)=e_6,\, \J(e_4)=e_5.$
				\item[] $\om=e^{12}+\left(\frac{\om_{45}(\alpha+1)}{\beta^2}\right)e^{34}+\left(\frac{(\omega_{34}\beta-\om_{45})(\alpha+1)}{\beta^2}\right)e^{36}+\frac{\om_{45}}{\beta^2}e^{45}\newline+\left(-\frac{\om_{45}(\alpha+1)}{\beta^2}\right)e^{56}$ with $\beta\neq0,\, \om_{45}>0,\, (\om_{34}\beta-\om_{45})(\alpha+1)>0, \, (\om_{34}\beta-\om_{45}-\om_{45}(\alpha+1))\om_{45}(\alpha+1)>0.$
			\end{itemize}
			\item $\R^2 \rtimes_1 \mathfrak{d}_{4,\frac12}:$
			\begin{itemize} \item []
			 $\left(\frac{\alpha}{2}16-t26,t16+\frac{\alpha}{2}26,\frac12 36,\frac12 46,-34+56,0\right)$
			
				\item[] $\J(e_1)=e_2,\, \J(e_3)=e_4,\, \J(e_5)=-e_6.$
				\item[] $\om=e^{12}+\tau (e^{34}-(\sigma+1)e^{56})$ with $\tau>0,\, \sigma+1>0.$
			\end{itemize}
			\item $\R^2 \rtimes_2 \mathfrak{d}_{4,\frac12}:$
			\begin{itemize} \item []
			 $\left(-\frac{3}{4}16-t26,t16-\frac{3}{4}26,\frac12 36,\frac12 46,-34+56,0\right).$
			
				\item[] $\J(e_1)=e_2,\, \J(e_3)=-e_4,\, \J(e_5)=-e_6.$
				\item[] $\om=e^{12}+\sigma(e^{34}+\frac12 e^{56})$ with $\sigma<0$.
			\end{itemize}
			\item $\R^2 \rtimes_2 \mathfrak{d}_{4,\frac12}:$
			\begin{itemize} \item []
			 $\left(-\frac{3}{4}16-t26,t16-\frac{3}{4}26,\frac12 36,\frac12 46,-34+56,0\right).$
			
				\item[] $\J(e_1)=e_2,\, \J(e_3)=2e_6,\, \J(e_4)=e_5.$
				\item[] $\om=e^{12}+\om_{22}e^{45}+\frac{\sigma}{2\om_{22}}e^{36}$ with $\om_{22}>0,\, \sigma>0.$
			\end{itemize}
			\item $\R^2 \rtimes_1 \mathfrak{d}'_{4,0}:$
			\begin{itemize} \item []
			 $\left(\frac{\mu}{2}16-t26,t16+\frac{\mu}{2}26,46,-36,-34,0\right)$
			
				\item[] $\J(e_1)=e_2,\, \J(e_3)=e_4,\, \J(e_5)=e_6.$
				\item[] $\om=e^{12}+\sigma(e^{34}-\mu e^{56})$ with $\sigma>0,\, \mu<0.$
			\end{itemize}
			\item $\R^2 \rtimes_2 \mathfrak{d}'_{4,0}:$
			\begin{itemize} \item []
			 $\left(\frac{\mu}{2}16-t26,t16+\frac{\mu}{2}26,46,-36,-34,0\right)$
			
				\item[] $\J(e_1)=e_2,\, \J(e_3)=e_4,\, \J(e_5)=-e_6.$
				\item[] $\om=e^{12}+\sigma(e^{34}-\mu e^{56})$ with $\sigma>0,\, \mu>0.$
			\end{itemize}
			\item $\R^2 \rtimes_1 \mathfrak{d}_{4,\lambda}$ ,\; $\lambda>\frac12, \lambda\neq1:$
			\begin{itemize}\item []
			 $\left(\left(\frac{\lambda}{2}-1\right)16-t26,t16+\left(\frac{\lambda}{2}-1\right)26,\lambda36,(1-\lambda)46,-34+56,0\right)$
			
				\item[] $\J(e_1)=e_2,\, \J(e_3)=-e_5,\, \J(e_4)=-\frac{1}{\lambda-1}e_6.$ 
				\item[] $\om=e^{12}-e^{35}-\sqrt{\sigma+\om_{22}}(\lambda-1)\,e^{36}-\sqrt{\sigma+\om_{22}}\,e^{45}+(1-\lambda)\om_{22}e^{46}$ with $\sigma<0, \,  \om_{22}\geq-\sigma$.
			\end{itemize}
			\newpage
			\item $\R^2 \rtimes_2 \mathfrak{d}_{4,\lambda}$ ,\; $\lambda>\frac12, \lambda\neq1:$
			\begin{itemize}\item []
			 $\left(\left(\frac{\lambda}{2}-1\right)16-t26,t16+\left(\frac{\lambda}{2}-1\right)26,\lambda36,(1-\lambda)46,-34+56,0\right)$
			\item[]  $\J(e_1)=e_2,\, \J(e_3)=-e_5,\, \J(e_4)=-\frac{1}{\lambda-1}e_6.$ 
			\item[] $\om_2=e^{12}-e^{35}+\sqrt{\sigma+\om_{22}}(\lambda-1)\,e^{36}+\sqrt{\sigma+\om_{22}}\,e^{45}+(1-\lambda)\om_{22}e^{46}$ with $\sigma<0,\, \om_{22}\geq-\sigma$.
			\end{itemize}
			\item $\R^2 \rtimes_3 \mathfrak{d}_{4,\lambda}$ ,\; $\lambda>\frac12, \lambda\neq1:$
				\begin{itemize}\item []
			 $\left(-\frac{\lambda+1}{2}16-t26,t16-\frac{\lambda+1}{2}26,\lambda36,(1-\lambda)46,-34+56,0\right)$
		
				\item[]  $\J(e_1)=e_2,\, \J(e_3)=-e_5,\, \J(e_4)=-\frac{1}{\lambda-1}e_6.$ 
				\item[] $\om=e^{12}-\om_{11}e^{35}-(\lambda-1)\om_{22}e^{46}$ with $\om_{11}>0,\, \om_{22}>0.$
			\end{itemize}
			\item $\R^2 \rtimes_4 \mathfrak{d}_{4,\lambda}$ ,\; $\lambda>\frac12, \lambda\neq1:$
			\begin{itemize} \item []
			 $\left(\left(\frac{\lambda}{2}-1\right)16-t26,t16+\left(\frac{\lambda}{2}-1\right)26,\lambda36,(1-\lambda)46,-34+56,0\right)$
			
				\item[] $\J(e_1)=e_2,\, \J(e_3)=\frac{1}{\lambda} e_6,\, \J(e_4)=e_5 $
				\item[] $\om=e^{12}+\om_{11}\lambda e^{36}+\om_{22}e^{45} $ with $\om_{11}>0,\, \om_{22}>0.$
			\end{itemize}
			\item $\R^2 \rtimes_1 \mathfrak{gl}_2:$
			\begin{itemize} \item []
			 $\left(-\frac{\mu}{2}16,-\frac{\mu}{2}26,-45,-2\times 34,2\times 35,0\right)$
			
				\item[] $\J(e_1)=e_2,\, \J(e_3)=-(e_4+e_5),\, \J(e_4)=\frac12 e_3-\frac{1}{\mu}e_6$ 	with $\mu \in \mathbb{R}\setminus \{0\}.$
				\item[] $\om=e^{12}+\om_{34}e^{34}+\om_{35}e^{35}+\om_{45}\mu e^{36}+\om_{45}e^{45}+\frac12 \om_{34}\mu e^{46}-\frac12 \om_{35}\mu e^{56}$ with 
				$\om_{34}\geq \om_{35}>0,\, \om_{45}\geq 0,\, \om_{34}\om_{35}-\om_{45}^2>0.$
			\end{itemize}
			\item $\R^2 \rtimes_2 \mathfrak{gl}_2:$
			\begin{itemize} \item []
			 $\left(\frac{\alpha}{2}16,\frac{\alpha}{2}26,-45,-2\times 34,2\times 35,0\right)$
			
				\item[] $\J(e_1)=e_2,\, \J(e_3)=-(e_4+e_5),\, \J(e_4)=\frac12 e_3-\frac{1}{\mu}e_6$ 	with $\mu \in \mathbb{R}\setminus \{0\}.$
				\item[] $\om=e^{12}+\om_{34}e^{34}+\om_{34}e^{35}-\frac12 \om_{34}\alpha e^{46}+\frac12 \om_{34}\alpha e^{56}$  with $\alpha\neq -\mu,\newline \om_{34}>0,\, \frac{\alpha}{\mu }<0.$
			\end{itemize}
			\item $\R^2 \rtimes_1 \mathfrak{u}_2:$
			\begin{itemize} \item []
			 $\left(\frac{\alpha}{2}16-t26,t16+\frac{\alpha}{2}26,45,-35,34,0\right)$
			
				\item[]  $\J(e_1)=e_2,\, \J(e_3)=ae_3-be_6,\, \J(e_4)=-e_5$ with $a\neq0,\, b\neq0.$
				\item[] $\om=e^{12}+\om_{45}\alpha e^{36}+\om_{45}e^{45}$  with  $\alpha\neq0,\, \om_{45}<0,\, \frac{\alpha}{b}>0.$
			\end{itemize}
			\item $\R^2 \rtimes_2 \mathfrak{u}_2:$
			\begin{itemize} \item []
			 $\left(-\frac{\alpha}{2}16+t26,-t16-\frac{\alpha}{2}26,45,-35,34,0\right)$
			
				\item[] $\J(e_1)=e_2,\, \J(e_3)=-be_6,\, \J(e_4)=-e_5$  with $b\neq0.$
				\item[] $\om=e^{12}+\om_{45}\alpha e^{36}+\om_{45}e^{45}$ with $\alpha\notin \{0,-\frac{1}{b}\},\, \om_{45}<0,\, \frac{\alpha}{b}>0.$
			\end{itemize}
			\item $\R^2 \rtimes_1 \mathfrak{d}'_{4,\delta}:$
				\begin{itemize} \item []
			 $\left(\frac{\mu}{2}16-t26,t16+\frac{\mu}{2}26,\frac{\delta}{2}36+46,-36+\frac{\delta}{2}46,-34+\delta56,0\right)$
		
				\item[] $\J(e_1)=e_2,\, \J(e_3)=-e_4,\, \J(e_5)=-e_6.$
				\item[] $\om=e^{12}+\sigma(e^{34}-(\delta+\mu )e^{56})$ with $\sigma<0,\, \delta+\mu <0,\, \mu \neq0.$
			\end{itemize}
			
			\item $\R^2 \rtimes_2 \mathfrak{d}'_{4,\delta}:$
				\begin{itemize} \item []
			 $\left(\frac{\mu}{2}16-t26,t16+\frac{\mu}{2}26,\frac{\delta}{2}36+46,-36+\frac{\delta}{2}46,-34+\delta56,0\right)$
		
				\item[] $\J(e_1)=e_2,\, \J(e_3)=-e_4,\, \J(e_5)=e_6.$
				\item[] $\om=e^{12}+\sigma(e^{34}-(\delta+\mu )e^{56})$ with $\sigma<0,\, \delta+\mu >0,\, \mu \neq0.$
			\end{itemize}
			\item $\R^2 \rtimes_3 \mathfrak{d}'_{4,\delta}:$
			\begin{itemize} \item []
			 $\left(\frac{\mu}{2}16-t26,t16+\frac{\mu}{2}26,\frac{\delta}{2}36+46,-36+\frac{\delta}{2}46,-34+\delta56,0\right)$
			
				\item[] $\J(e_1)=e_2,\, \J(e_3)=e_4,\, \J(e_5)=e_6.$
				\item[] $\om=e^{12}+\sigma(e^{34}-(\delta+\mu )e^{56})$ with $\sigma>0,\, \delta+\mu <0,\, \mu \neq0.$
			\end{itemize}
			\item $\R^2 \rtimes_4 \mathfrak{d}'_{4,\delta}:$
			\begin{itemize} \item []
			 $\left(\frac{\mu}{2}16-t26,t16+\frac{\mu}{2}26,\frac{\delta}{2}36+46,-36+\frac{\delta}{2}46,-34+\delta56,0\right)$
				\item[] $\J(e_1)=e_2,\, \J(e_3)=e_4,\, \J(e_5)=-e_6.$
				\item[] $\om=e^{12}+\sigma(e^{34}-(\delta+\mu )e^{56})$ with $\sigma>0,\, \delta+\mu >0,\, \mu \neq0.$
			\end{itemize}
		\end{itemize}
	\end{theo}
	
	\begin{proof}
		We will give the proof in the case $\mathbf{\mathbb{R}^2 \Join \mathfrak{rr}_{3,1}}$ since all cases should be handled in a similar way. Let $(\mathfrak{rr}_{3,1}=span\{e_3,e_4,e_5,e_6\},\om_2,\J_2)$ be the  four-dimensional Hermitian Lie algebra  with $[e_3,e_4]=e_4$, $[e_3,e_5]=e_5$, $\om_2=\sigma e^{36}+e^{45}\, (\sigma>0)$ and $\J_2(e_3)=e_6$, $\J_2(e_4)=e_5$ and
		$\theta_2=-2e^3$. 
		Using the definitions above, we know that $\rho_1 : \R^2\lr Der(\mathfrak{rr}_{3,1})\esp \rho_2 : \mathfrak{rr}_{3,1}\lr Der(\R^2).$
		Lets  look now for the derivations $\rho_1(e_1), \, \rho_1(e_2)$ of $\mathfrak{rr}_{3,1}$ which commute with $\J_2$ and $\rho_2(e_3), \, \rho_2(e_4), \, \rho_2(e_5), \, \rho_2(e_6)$ of $\R^2$ which commute with $\J_1$.
		We get
		\begin{center}
			$\rho_2(e_3)=\begin{pmatrix}
				x_2 & -x_1 
				\\
				x_1 & x_2 
			\end{pmatrix}
			$,\; $\rho_2(e_4)=\begin{pmatrix}
				y_2 & -y_1 
				\\
				y_1 & y_2 
			\end{pmatrix}
			$,\; $\rho_2(e_5)=\begin{pmatrix}
				z_2 & -z_1 
				\\
				z_1 & z_2 
			\end{pmatrix}
			$,\; $\rho_2(e_6)=\begin{pmatrix}
				t_2 & -t_1 
				\\
				t_1 & t_2 
			\end{pmatrix}
			$
		\end{center}
		and
		\begin{center}
			$\rho_1(e_1)= \begin{pmatrix}
				0 & 0 & 0 & 0 
				\\
				0 & a_{1} & -a_{2} & 0 
				\\
				0 & a_{2} & a_{1} & 0 
				\\
				0 & 0 & 0 & 0 
			\end{pmatrix}
			$,\;
			$\rho_1(e_2)= \begin{pmatrix}
				0 & 0 & 0 & 0 
				\\
				0 & b_{1} & -b_{2} & 0 
				\\
				0 & b_{2} & b_{1} & 0 
				\\
				0 & 0 & 0 & 0 
			\end{pmatrix}$
		\end{center}
		with $x_1,x_2,y_1,y_2,z_1,z_2,t_1,t_2,a_1,a_2,b_1,b_2\in \mathbb{R}$.\\
		Taking into consideration that $\rho_1$ and $\rho_2$ are representations, we have:
		\begin{center}
			$\rho_2(e_3)=\begin{pmatrix}
				x_2 & -x_1\\
				x_1 & x_2 
			\end{pmatrix}
			$,\; $\rho_2(e_4)=0
			$,\; $\rho_2(e_5)=0
			$,\; $\rho_2(e_6)=\begin{pmatrix}
				t_2 & -t_1 \\
				t_1 & t_2 
			\end{pmatrix}$.
		\end{center}
		Since $\theta_1=0$ and $\theta_2=-2e^3$, we have : $0=\theta_1(e_1)=tr_2(\rho_1(e_1))=2a_1$ and $0=\theta_1(e_2)=tr_2(\rho_1(e_2))=2b_1$ and  $-2=\theta_2(e_3)
		=tr_1(\rho_2(e_3))
		=2x_2$ and $0=\theta_2(e_6)
		=tr_1(\rho_2(e_6))
		=2t_2$.\\
		So \hspace{4.5cm} $a_1=b_1=0$ and $x_2=-1$ and $t_2=0$.\\
		As a consequence, we get :
		\begin{center}
			$\rho_2(e_3)=\begin{pmatrix}
				-1 & -x_1 
				\\
				x_1 & -1 
			\end{pmatrix}
			$,\; $\rho_2(e_4)=0
			$,\; $\rho_2(e_5)=0
			$,\; $\rho_2(e_6)=\begin{pmatrix}
				0 & -t_1 
				\\
				t_1 & 0 
			\end{pmatrix}
			$
		\end{center}
		and
		\begin{center}
			$\rho_1(e_1)= \begin{pmatrix}
				0 & 0 & 0 & 0 
				\\
				0 & 0 & -a_{2} & 0 
				\\
				0 & a_{2} & 0 & 0 
				\\
				0 & 0 & 0 & 0 
			\end{pmatrix}
			$,\;
			$\rho_1(e_2)=\begin{pmatrix}
				0 & 0 & 0 & 0 
				\\
				0 & 0 & -b_{2} & 0 
				\\
				0 & b_{2} & 0 & 0 
				\\
				0 & 0 & 0 & 0 
			\end{pmatrix}$.
		\end{center}
		A calculation show that the non-zero brackets on the balanced Hermitian Lie algebra $\mathbb{R}^2 \Join \mathfrak{rr}_{3,1}$ are defined by :
		\begin{center}
			$[e_1,e_3]=e_1-x_1e_2, \quad [e_1,e_6]=-t_1e_2, \quad [e_2,e_3]=x_1e_1+e_2, \quad [e_2,e_6]=t_1e_1, \newline [e_3,e_4]=e_4, \quad [e_3,e_5]=e_5$,
		\end{center}
		and result follows.		
	\end{proof}
	
	\subsection{$\mathrm{aff(\R)\Join\G_2}$  with  $\dim(\G)=4$}
	Let  $\mathrm{aff}(\R)=span \{e_1,e_2\}$ with $[e_1,e_2]=e_1$, $\om_1=e^{12}$ and  $\J_1(e_1)=e_2$ and $(\mathfrak{g},\om_2,\J_2)$ be a four-dimensional Hermitian Lie algebra, 
	The Lie algebra $\G$ is of dimension four, so its associated Lee form $\theta_2$ satisfy $\md \om_2=\theta_2 \wedge \om_2$. In addition to that $\theta_2=\tr_1(\rho_2(.))$ is a closed 1-form. As a consequence $(\G,\J_2,\K_2)$ is LCK. On the other hand, LCK structures on four-dimensional Lie algebras are classified in \cite{DA}. We investigate this classification by twisting  $\mathrm{aff(\R)}$ with each  of its Lie algebras.
	\begin{theo}\label{Th3.2}
		Balanced Hermitian twisted cartesian products Lie algebras of type $(\mathrm{aff}(\R)\Join\G,\omega,\J )$ where $\dim(\mathfrak{g)}=4$ are described as follows:
		\begin{itemize}
			\item $\mathrm{aff}(\R) \rtimes_1 \mathfrak{d}_{4,2}:$
			\begin{itemize} \item []
			 $(-12,0,2\times36,-46,56-34,0)$
			
				\item [] $\J(e_1)=e_2,\, \J(e_3)=-e_5,\, \J(e_4)=-e_6.$
				\item [] $\om=e^{12}-e^{35}-\sqrt{\sigma+\om_{22}}\,e^{36}-\sqrt{\sigma+\om_{22}}\,e^{45}-\om_{22}e^{46}$ with $\sigma<0,\newline   \om_{22}\geq-\sigma$.
			\end{itemize}
			\item $\mathrm{aff}(\R) \rtimes_2 \mathfrak{d}_{4,2}:$
			\begin{itemize} \item []
			 $(-12,0,2\times36,-46,56-34,0)$
			
				\item [] $\J(e_1)=e_2,\, \J(e_3)=-e_5,\, \J(e_4)=-e_6.$
				\item [] $\om=e^{12}-e^{35}+\sqrt{\sigma+\om_{22}}\,e^{36}+\sqrt{\sigma+\om_{22}}\,e^{45}-\om_{22}e^{46}$ with $\sigma<0,\newline \om_{22}\geq-\sigma$.
			\end{itemize}
			\item $\mathrm{aff}(\R) \rtimes_3 \mathfrak{d}_{4,2}:$
			\begin{itemize}\item []
			 $(-12,0,2\times36,-46,56-34,0)$
			
				\item [] $\J(e_1)=e_2,\, \J(e_3)=\frac12 e_6,\, \J(e_4)=e_5.$
				\item [] $\om=e^{12}+2\om_{11}e^{36}+\om_{22}e^{45}$ with $\om_{11}>0,\, \om_{22}>0.$
			\end{itemize}
		\end{itemize}
	\end{theo}
	\begin{proof}
		The proof is similar to Theorem \ref{Th3.1}.
	\end{proof}
	\begin{conclu}
		By Theorem \ref{Th3.1} we conclude that all the Hermitian twisted cartesian products $\R^2\Join \G$ carries a balanced structure, contrary to Theorem \ref{Th3.2} the $\mathrm{aff}\Join \mathfrak{d}_{4,2}$ is the only one.
	\end{conclu}

\end{document}